\documentclass[runningheads]{llncs}

\usepackage{amsfonts,amsmath,amssymb}
\usepackage{mathtools}
\usepackage{graphicx}
\usepackage{fixmath}
\usepackage[hidelinks]{hyperref}
\usepackage{tikz}
\usepackage{cite}

\usepackage{mathtools}

\newtheorem{theo}{Theorem}

\def\qedsymbol{\ensuremath{\scriptstyle\blacksquare}}

\newcommand{\Prob}{\mathbb{P}}
\newcommand{\Ex}{\mathbb{E}}
\newcommand{\Var}{\mathbb{V}}

\newcommand{\SSet}{\mathcal{S}}
\newcommand{\bst}{\operatorname{bst}}

\newcommand{\sym}{\operatorname{sym}}
\DeclareMathOperator{\Aut}{\mathrm{Aut}}
\DeclarePairedDelimiter{\abs}{\lvert}{\rvert}

\begin{document}
\title{On the Collection of Fringe Subtrees in Random Binary Trees}

\author{Louisa Seelbach Benkner\inst{1}\thanks{This project has received funding from the European Union’s Horizon 2020 research and innovation programme under the Marie Sk\l odowska-Curie grant agreement No 731143 and the DFG research project LO 748/10-1 (QUANT-KOMP).} \and
Stephan Wagner\inst{2,3} }
\institute{Department f\"ur Elektrotechnik und Informatik, Universit\"at Siegen, H\"olderlinstrasse~3, 57076 Siegen, Germany \\
\email{seelbach@eti.uni-siegen.de}
\and 
Department of Mathematical Sciences, Stellenbosch University, Private Bag X1, Matieland 7602, South Africa \\
\email{swagner@sun.ac.za}
\and
Department of Mathematics, Uppsala Universitet, Box 480, 751 06 Uppsala, Sweden \\
\email{stephan.wagner@math.uu.se}
}
\maketitle              
\begin{abstract}
A fringe subtree of a rooted tree is a subtree consisting of one of the nodes and all its descendants. In this paper, we are specifically interested in the number of non-isomorphic trees that appear in the collection of all fringe subtrees of a binary tree. This number is analysed under two different random models: uniformly random binary trees and random binary search trees.
 
In the case of uniformly random binary trees, we show that the number of non-isomorphic fringe subtrees lies between $c_1n/\sqrt{\ln n}(1+o(1))$ and $c_2n/\sqrt{\ln n}(1+o(1))$ for two constants $c_1 \approx 1.0591261434$ and $c_2 \approx 1.0761505454$, both in expectation and with high probability, where $n$ denotes the size (number of leaves) of the uniformly random binary tree. A similar result is proven for random binary search trees, but the order of magnitude is $n/\ln n$ in this case.
 
Our proof technique can also be used to strengthen known results on the number of distinct fringe subtrees (distinct in the sense of ordered trees). This quantity is of the same order of magnitude in both cases, but with slightly different constants in the upper and lower bounds.

\end{abstract}

\section{Introduction}

A subtree of a rooted tree that consists of a node and all its descendants is called a \emph{fringe subtree}. Fringe subtrees are a natural object of study in the context of random trees, and there are numerous results for various random tree models, see e.g. \cite{aldous91, dennertgr10, devroye14, FengM10}.

Fringe subtrees are of particular interest in computer science: One of the most important and widely used lossless compression methods for rooted trees is to represent a tree as a directed acyclic graph, which is obtained by merging nodes that are roots of identical fringe subtrees. This compressed representation of the tree is often shortly referred to as \emph{minimal DAG} and its size (number of nodes) is the number of distinct fringe subtrees occurring in the tree. Compression by minimal DAGs has found numerous applications in various areas of computer science, as for example in compiler construction \cite[Chapter~6.1 and 8.5]{AhoSU86}, unification \cite{PatersonW78}, symbolic model checking (binary decision diagrams) \cite{Bry92}, information theory \cite{GanardiHLS19, ZhangYK14} and XML compression and querying \cite{BuGrKo03,FrGrKo03}. 

In this work, we investigate the number of fringe subtrees in random binary trees, i.e. random trees such that each node has either exactly two or no children. 
So far, this problem has mainly been studied with respect to ordered fringe subtrees in random ordered binary trees: A \emph{uniformly random ordered binary tree} of size $n$ (with $n$ leaves) is a random tree whose probability distribution is the uniform probability distribution on the set of ordered binary trees of size~$n$. In \cite{FlajoletSS90}, Flajolet, Sipala and Steyaert proved that the expected number of distinct ordered fringe subtrees in a uniformly random ordered binary tree of size $n$ is asymptotically equal to $c \cdot n/\sqrt{\ln n}$, where $c$ is the constant $2\sqrt{\ln 4/\pi}$. 
This result of Flajolet et al. was extended to unranked labelled trees in \cite{MLMN13} (for a different constant $c$). Moreover, an alternative proof to the result of Flajolet et al. was presented in \cite{RalaivaosaonaW15} in the context of simply-generated families of trees. 

Another important type of random trees are so-called \emph{random binary search trees}:
A random binary search tree of size $n$ is a binary search tree built by inserting the keys $\{1, \dots, n\}$ according to a uniformly chosen random permutation on $\{1, \dots, n\}$. Random binary search trees naturally arise in theoretical computer science, see e.g. \cite{Drmota09}. In \cite{FlajoletGM97}, Flajolet, Gourdon and Martinez proved that the expected number of distinct ordered fringe subtrees in a random binary search tree of size $n$ is $O(n/\ln n)$. This result was improved in \cite{Devroye98} by Devroye, who showed that the asymptotics $\Theta(n/\ln n)$ holds.
Moreover, the result of Devroye was generalized from random binary search trees to a broader class of random ordered binary trees in \cite{SeelbachLo18}, where the problem of estimating the expected number of distinct ordered fringe subtrees in random binary trees was considered in the context of so-called leaf-centric binary tree sources, which were introduced in \cite{KiefferYS09, ZhangYK14}  as a general framework for modeling probability distributions on the set of ordered binary trees of size $n$.

In this work, we focus on estimating the number of \emph{non-isomorphic} fringe subtrees in random ordered binary trees, where we call two binary trees non-isomorphic if they are distinct as unordered binary trees. This question arises quite naturally for example in the context of XML compression: Here, one distinguishes between so-called document-centric XML, for which the corresponding XML document trees are ordered, and data-centric XML, for which the corresponding XML document trees are unordered. Understanding the interplay between ordered and unordered structures has thus received considerable attention in the context of XML (see, for example, \cite{AbiteboulBV15, BonevaCS15, ZhangDW15}). In particular, in \cite{LohreyMR17}, it was investigated whether tree compression can benefit from unorderedness. For this reason, so-called \emph{unordered minimal DAGs} were considered. An unordered minimal DAG of a binary tree is a directed acyclic graph obtained by merging nodes that are roots of isomorphic fringe subtrees, i.e. of fringe subtrees which are identical as unordered trees. From such an unordered minimal DAG, an unordered representation of the original tree can be uniquely retrieved. The size of this compressed representation is the number of non-isomorphic fringe subtrees occurring in the tree. So far, only some worst-case estimates comparing the size of a minimal DAG to the size of its corresponding unordered minimal DAG are known: Among other things, it was shown in \cite{LohreyMR17} that the size of an unordered minimal DAG of a binary tree can be exponentially smaller than the size of the corresponding (ordered) minimal DAG. 

However, no average-case estimates comparing the size of the minimal DAG of a binary tree to the size of the corresponding unordered minimal DAG are known so far. In particular, in \cite{LohreyMR17} it is stated as an open problem to estimate the expected number of non-isomorphic fringe subtrees in a uniformly random ordered binary tree of size $n$ and conjectured that this number asymptotically grows as $\Theta(n/\sqrt{\ln n})$.
 
In this work, as one of our main theorems, we settle this open conjecture by proving upper and lower bounds of order $n/\sqrt{\ln n}$ for the number of non-isomorphic fringe subtrees which hold both in expectation and with high probability (i.e., with probability tending to $1$ as $n \to \infty$). Our approach can also be used to obtain an analogous result for random binary search trees, though the order of magnitude changes to $\Theta(n/\ln n)$. Again, we have upper and lower bounds in expectation and with high probability. Our two main theorems read as follows.

\begin{theo}\label{thm:unordereduniform}
Let $F_n$ be the total number of non-isomorphic fringe subtrees in a uniformly random ordered binary tree with $n$ leaves. For two constants $c_1 \approx 1.0591261434$ and $c_2 \approx 1.0761505454$, the following holds:
\begin{enumerate}
\item[(i)] $\displaystyle c_1 \frac{n}{\sqrt{\ln n}} (1+o(1)) \leq \Ex(F_n) \leq c_2 \frac{n}{\sqrt{\ln n}} (1+o(1))$,
\item[(ii)] $\displaystyle c_1 \frac{n}{\sqrt{\ln n}} (1+o(1)) \leq F_n \leq c_2 \frac{n}{\sqrt{\ln n}} (1+o(1))$ with high probability.
\end{enumerate}

\end{theo}

%Here, and throughout the rest of the paper, $\ln x$ denotes the natural logarithm of a positive real number $x$.
%We also consider the corresponding problem of counting the number of non-isomorphic fringe subtrees in random binary search trees. 

\begin{theo}\label{thm:unorderedbst}
Let $G_n$ be the total number of non-isomorphic fringe subtrees in a random binary search tree with $n$ leaves. For two constants $c_3 \approx 1.5470025923$ and $c_4 \approx 1.8191392203$, the following holds:
\begin{enumerate}
\item[(i)] $\displaystyle c_3 \frac{n}{\ln n} (1+o(1)) \leq \Ex(G_n) \leq c_4 \frac{n}{\ln n} (1+o(1))$,
\item[(ii)] $\displaystyle c_3 \frac{n}{\ln n} (1+o(1)) \leq G_n \leq c_4 \frac{n}{\ln n} (1+o(1))$ with high probability.
\end{enumerate}

\end{theo}

To prove the above Theorems \ref{thm:unordereduniform} and \ref{thm:unorderedbst}, we refine techniques from \cite{RalaivaosaonaW15}. Our proof technique also applies to the problem of estimating the number of distinct ordered fringe subtrees in uniformly random binary trees or in random binary search trees. In this case, upper and lower bounds for the expected value have already been proven by other authors. Our new contribution is to show that they also hold with high probability.

% Thus, as a side result of our work, we are able to improve the result by Flajolet et al. in \cite{FlajoletSS90}. 

\begin{theo}\label{thm:ordereduniform}
Let $H_n$ denote the total number of distinct fringe subtrees in a uniformly random ordered binary tree with $n$ leaves. Then, for the constant $c =2\sqrt{\ln 4/\pi} \approx 1.3285649405$, the following holds:
\begin{enumerate}
\item[(i)] $\displaystyle \Ex(H_n) = c \frac{n}{\sqrt{\ln n}}(1+o(1))$,
\item[(ii)] $\displaystyle H_n = c \frac{n}{\sqrt{\ln n}}(1+o(1))$ with high probability.
\end{enumerate}
\end{theo}
Here, the first part (i) was already shown in \cite{FlajoletSS90} and \cite{RalaivaosaonaW15}, part (ii) is new. Similarly, we are able to strengthen the results of \cite{Devroye98} and \cite{SeelbachLo18}:

\begin{theo}\label{thm:orderedbst}
Let $J_n$ be the total number of distinct fringe subtrees in a random binary search tree with $n$ leaves. For two constants $c_5 \approx 2.4071298335$ and $c_6 \approx 2.7725887222$, the following holds:
\begin{enumerate}
\item[(i)] $\displaystyle c_5 \frac{n}{\ln n} (1+o(1)) \leq \Ex(J_n) \leq c_6 \frac{n}{\ln n} (1+o(1))$,
\item[(ii)] $\displaystyle c_5 \frac{n}{\ln n} (1+o(1)) \leq J_n \leq c_6 \frac{n}{\ln n} (1+o(1))$ with high probability.
\end{enumerate}

\end{theo}

The upper bound in part (i) can already be found in \cite{FlajoletGM97} and \cite{Devroye98}. Moreover, a lower bound of the form $\Ex(J_n)\geq \frac{\alpha n}{\ln n}(1+o(1))$ was already shown in \cite{Devroye98} for the constant $\alpha=(\ln 3)/2 \approx 0.5493061443$ and in \cite{SeelbachLo18} for the constant $\alpha \approx 0.6017824584$. So our new contributions are part (ii) and the improvement of the lower bound on $\Ex(J_n)$.

\section{Preliminaries}
Let $\mathcal{T}$ denote the set of ordered binary trees, i.e. of ordered rooted trees such that each node has either exactly two or no children. We define the \emph{size} $|t|$ of a binary tree $t \in \mathcal{T}$ as the number of leaves of $t$ and by $\mathcal{T}_k$ we denote the set of binary trees of size $k$ for every integer $k \geq 1$. 
It is well known that $|\mathcal{T}_k|=C_{k-1}$, where $C_{k}$ denotes the $k$-th \emph{Catalan number} \cite{FlajoletS09}: We have
\begin{align}\label{eq:Cat}
C_k = \frac{1}{k+1}\binom{2k}{k} \sim \frac{4^k}{\sqrt{\pi}k^{3/2}}(1+O(1/k)),
\end{align}
where the asymptotic growth of the Catalan numbers follows from Stirling's Formula \cite{FlajoletS09}. Analogously, let $\mathcal{U}$ denote the set of unordered binary trees, i.e. of unordered rooted trees such that each node has either exactly two or no children. The \emph{size} $|u|$ of an unordered tree $u \in \mathcal{U}$ is again the number of leaves of $u$ and by $\mathcal{U}_k$ we denote the set of unordered binary trees of size $k$. We have
$|\mathcal{U}_k| = W_k$, where $W_k$ denotes the $k$-th \emph{Wedderburn-Etherington number}. Their asymptotic growth is 
\begin{align}\label{eq:wedder}
W_k \sim A \cdot k^{-3/2} \cdot b^k,
\end{align}
for certain positive constants $A, b$ \cite{bonaf09, finch03}. In particular, we have $b \approx 2.4832535362$.

A  \emph{fringe subtree} of a binary tree is a subtree consisting of a node and all its descendants. For a binary tree $t$ and a given node $v \in t $, let $t(v)$ denote the fringe subtree of $t$ rooted at $v$. Two fringe subtrees are called \emph{distinct} if they are distinct as ordered binary trees. 

Every tree $t \in \mathcal{T}$ can be considered as an element of $\mathcal{U}$ by simply forgetting the ordering on $t$'s nodes. If two binary trees $t_1, t_2$ correspond to the same unordered tree $u \in \mathcal{U}$, we call them \emph{isomorphic}: Thus, we obtain  a partition of $\mathcal{T}$ into isomorphism classes. If two binary trees $t_1, t_2 \in \mathcal{T}$ belong to the same isomorphism class, we can obtain $t_1$ from $t_2$ and vice versa by reordering the children of some of $t_1$'s (respectively, $t_2$'s) inner nodes. 
An inner node $v$ of an ordered or unordered binary tree $t$ is called a \emph{symmetrical node} if the fringe subtrees rooted at $v$'s children are isomorphic. Let $\sym(t)$ denote the number of symmetrical nodes of $t$. The cardinality of the automorphism group of $t$ is given by $\abs{\Aut (t)} = 2^{\sym (t)}$. Thus, by the orbit-stabilizer theorem, there are $2^{k-1-\sym(t)}$ many ordered binary trees in the isomorphism class of $t \in \mathcal{T}_k$, and likewise $2^{k-1-\sym(t)}$ many ordered representations of $t \in \mathcal{U}_k$.

We consider two types of probability distributions on the set of ordered binary trees of size $n$:

\begin{itemize}
\item[(i)] The \emph{uniform probability distribution} on $\mathcal{T}_n$, that is, every binary tree of size $n$ is assigned the same probability $\frac{1}{C_{n-1}}$. A random variable taking values in $\mathcal{T}_n$ according to the uniform probability distribution is called a \emph{uniformly random (ordered) binary tree} of size $n$.
\item[(ii)] The probability distribution induced by the so-called \emph{Binary Search Tree Model} (see e.g. \cite{Drmota09, FlajoletGM97}): The corresponding probability mass function $P_{\bst}: \mathcal{T}_n \rightarrow [0,1]$ is given by
\begin{align}\label{def:pbst}
P_{\bst}(t) = \prod_{v \in t \atop |t(v)| >1} \frac{1}{|t(v)|-1},
\end{align}
for every $n \geq 1$. %The function $P_{\bst}$ satisfies an easy recurrence relation: Let $t_1$ and $t_2$ denote the fringe subtrees rooted at the children of the root node of a binary tree $t$. Then
%\begin{align*}
%P_{\bst}(t) = \begin{cases} 1 \quad &\text{if } |t| = 1, \\
%\frac{1}{|t|-1}P_{\bst}(t_1)P_{\bst}(t_2) \quad &\text{otherwise.}
%\end{cases} 
%\end{align*}
A random variable taking values in $\mathcal{T}_n$ according to this probability mass function is called a \emph{random binary search tree} of size $n$. 
\end{itemize}

Before we start with proving our main results, we need two preliminary lemmas on the number of fringe subtrees in uniformly random ordered binary trees and in random binary search trees: 

\begin{lemma}\label{lem:fs_lemma}
Let $a,\varepsilon$ be positive real numbers with $\varepsilon < \frac13$. For every positive integer $k$ with $a \ln n \leq k \leq n^{\varepsilon}$, let $\SSet_k \subset \mathcal{T}_k$ be a set of ordered binary trees with $k$ leaves. We denote the cardinality of $\SSet_k$ by $s_k$. Let $X_{n,k}$ denote the (random) number of fringe subtrees with $k$ leaves in a uniformly random ordered binary tree with $n$ leaves that belong to $\SSet_k$. Moreover, let $Y_{n,\varepsilon}$ denote the (random) number of arbitrary fringe subtrees with more than $n^{\varepsilon}$ leaves in a uniformly random ordered binary tree with $n$ leaves. We have

\begin{enumerate}
\item[(1)] $\Ex(X_{n,k}) = s_k 4^{1-k} n \big(1+O(k/n) \big)$ for all $k$ with $a \ln n \leq k \leq n^{\varepsilon}$, the $O$-constant being independent of $k$,
\item[(2)]$\Var(X_{n,k}) = s_k 4^{1-k} n(1+O(k^{-1/2}))$ for all $k$ with $a \ln n \leq k \leq n^{\varepsilon}$, again with an $O$-constant that is independent of $k$,
\item[(3)]$\Ex(Y_{n,\varepsilon}) = O(n^{1-\varepsilon/2})$ and
\item[(4)] with high probability, the following statements hold simultaneously:
\begin{itemize}
\item[(i)] $|X_{n,k} - \Ex(X_{n,k}) | \leq s_k^{1/2}2^{-k} n^{1/2+\varepsilon}$ for all $k$ with $a \ln n \leq k \leq n^{\varepsilon}$,
\item[(ii)] $Y_{n,\varepsilon} \leq n^{1-\varepsilon/3}$.
\end{itemize}
\end{enumerate}

\end{lemma}

We emphasize (since it will be important later) that the inequality in part (4), item (i), does not only hold with high probability for each individual $k$, but that it is satisfied with high probability for all $k$ in the given range simultaneously.

\begin{proof}
(1) Recall first that the number of ordered binary trees with $n$ leaves is the Catalan number $C_{n-1}=\frac{1}{n} \binom{2n-2}{n-1}$. We observe that every occurrence of a fringe subtree in $\SSet_k$ in a tree with $n$ leaves can be obtained by choosing an ordered tree with $n-k+1$ leaves, picking one of the leaves and replacing it by a tree in $\SSet_k$. Thus the total number of occurrences is
$$\frac{1}{n-k+1} \binom{2n-2k}{n-k} \cdot (n-k+1) \cdot s_k = \binom{2n-2k}{n-k} s_k.$$
Consequently, the average number is
$$\Ex(X_{n,k}) = \frac{\binom{2n-2k}{n-k} s_k}{\frac{1}{n} \binom{2n-2}{n-1}} = s_k 4^{1-k} n \big( 1 + O(k/n) \big),$$
by Stirling's formula (the $O$-constant being independent of $k$ in the indicated range).
\\\\
\noindent
(2) The variance is determined in a similar fashion: we first count the total number of pairs of fringe subtrees in $\SSet_k$ that appear in the same ordered tree with $n$ leaves. Each such pair can be obtained as follows: take an ordered tree with $n-2k+2$ leaves, pick two leaves, and replace them by fringe subtrees in $\SSet_k$. The total number is thus
$$\frac{1}{n-2k+2} \binom{2n-4k+2}{n-2k+1} \cdot \binom{n-2k+2}{2} \cdot s_k^2 = \frac{n-2k+1}{2} \binom{2n-4k+2}{n-2k+1} s_k^2,$$
giving us
$$\Ex \Big( \binom{X_{n,k}}{2} \Big) = \frac{\frac{n-2k+1}{2} \binom{2n-4k+2}{n-2k+1} s_k^2}{\frac{1}{n} \binom{2n-2}{n-1}} = s_k^2 4^{2-2k} \frac{n^2}{2} \big( 1 + O(k/n) \big),$$
again by Stirling's formula. The second moment and the variance are now derived from this formula in a straightforward fashion: We find
\begin{align*}
\Ex(X_{n,k}^2) = 2\Ex\Big(\binom{X_{n,k}}{2}\Big)+\Ex(X_{n,k}) = \left(s_k^2 4^{2-2k}n^2+s_k4^{1-k}n \right)\big(1+O(k/n)\big),
\end{align*}
and thus, as $s_k/C_{k-1}\leq 1$,
\begin{align*}
\Var(X_{n,k})&=\Ex(X_{n,k}^2)-\Ex(X_{n,k})^2 = O(s_k^24^{2-2k}nk)+s_k4^{1-k}n(1+O(k/n))\\ &= O\left(\frac{s_k^2}{C_{k-1}^2}\frac{n}{k^2}\right) + \frac{s_k}{C_{k-1}}\frac{n}{\sqrt{\pi}k^{3/2}}(1+O(1/k)) = s_k 4^{1-k}n(1+O(1/k^{1/2})).
\end{align*}

\noindent
(3) To obtain the estimate for $\Ex(Y_{n,\varepsilon})$, we observe that the average total number of fringe subtrees with $k$ leaves is
$$\frac{\binom{2n-2k}{n-k} \cdot \frac{1}{k} \binom{2k-2}{k-1}}{\frac{1}{n} \binom{2n-2}{n-1}} = O \Big( \frac{n^{3/2}}{k^{3/2} (n-k+1)^{1/2}} \Big),$$
where the estimate follows from Stirling's formula again for $k > n^{\varepsilon}$. Summing over all $k$, we get
$$\Ex(Y_{n,\varepsilon}) = O \Big(n^{3/2} \sum_{n^{\varepsilon}<k \leq n} \frac{1}{k^{3/2}(n-k+1)^{1/2} } \Big) = O(n^{1-\varepsilon/2}).$$
\\
\noindent
(4) For the second part, we apply Chebyshev's inequality to obtain concentration of $X_{n,k}$:
$$\Prob \Big( \big|X_{n,k} - \Ex(X_{n,k}) \big| \geq s_k^{1/2} 2^{-k} n^{1/2+\varepsilon} \Big) \leq \frac{\Var(X_{n,k})}{s_k 4^{-k}  n^{1+2\varepsilon}} = O\big(n^{-2\varepsilon} \big).$$
Hence, by the union bound, the probability that the stated inequality fails for any $k$ in the given range is only $O(n^{-\varepsilon})$, proving that the first statement holds with high probability. Finally, Markov's inequality implies that
$$\Prob \big( Y_{n,\varepsilon} > n^{1-\varepsilon/3} \big) \leq \frac{\Ex(Y_{n,\varepsilon})}{n^{1-\varepsilon/3}} = O(n^{-\varepsilon/6}),$$
showing that the second inequality holds with high probability as well. \hfill \qedsymbol
\end{proof}

For the number of fringe subtrees in random binary search trees, a very similar lemma holds:

\begin{lemma}\label{lemma1bst}
Let $a, \varepsilon$ be positive real numbers with $\varepsilon < \frac{1}{3}$ and let $n$ and $k$ denote positive integers. Moreover, for every $k$, let $\mathcal{S}_k \subset \mathcal{T}_k$ be a set of ordered binary trees with $k$ leaves and let $p_k$ denote the probability that a random binary search tree is contained in $\mathcal{S}_k$, that is, $p_k = \sum P_{\bst}(t)$, where the sum is taken over all binary trees in $\mathcal{S}_k$. Let $X_{n,k}$ denote the (random) number of fringe subtrees with $k$ leaves in a random binary search tree with $n$ leaves that belong to $\mathcal{S}_k$. Moreover, let $Y_{n, \varepsilon}$ denote the (random) number of arbitrary fringe subtrees with more than $n^{\varepsilon}$ leaves in a random binary search tree with $n$ leaves. We have
\begin{itemize}
\item[(1)] $\Ex(X_{n,k}) = \frac{2p_kn}{k(k+1)}$ for $1 \leq k < n$,
\item[(2)] $\Var(X_{n,k}) = O(p_kn/k^2)$ for all $k$ with $a\ln n \leq k \leq n^{\varepsilon}$, where the $O$-constant is independent of $k$,
\item[(3)] $\Ex(Y_{n, \varepsilon}) = 2n/\lceil n^{\varepsilon}\rceil-1 = O(n^{1-\varepsilon})$ and
\item[(4)] with high probability, the following statements hold simultaneously: 
\begin{itemize}
\item[(i)] $|X_{n,k}-\Ex(X_{n,k})| \leq p_k^{1/2}k^{-1}n^{1/2+\varepsilon}$ for all $k$ with $a\ln n \leq k \leq n^{\varepsilon}$,
\item[(ii)] $Y_{n, \varepsilon}\leq n^{1-\varepsilon/2}$.
\end{itemize}
\end{itemize}

\end{lemma}

\begin{proof}
(1) In order to estimate $\Ex(X_{n,k})$, we define $Z_{n,k}$ as the (random) number of arbitrary fringe subtrees with $k$ leaves in a random binary search tree with $n$ leaves. That is, $Z_{n,k}=X_{n,k}$ for $\mathcal{S}_k = \mathcal{T}_k$. Applying the law of total expectation, we find
\begin{align*}
\Ex(X_{n,k}) = \sum_{m=0}^n \Ex(X_{n,k} \mid Z_{n,k} = m)\mathbb{P}(Z_{n,k}=m).
\end{align*}
As $X_{n,k}$ conditioned on $Z_{n,k}=m$ for some integer $m$ is binomially distributed with parameters $m$ and $p_k$, we find $\Ex(X_{n,k} \mid Z_{n,k}=m)=mp_k$ and hence
\begin{align*}
\Ex(X_{n,k}) = p_k \sum_{m=0}^nm \mathbb{P}(Z_{n,k}=m) = p_k \Ex(Z_{n,k}).
\end{align*}
With
$
\Ex(Z_{n,k}) =\frac{2n}{k(k+1)}
$
(see for example \cite{FengMP08}), the statement follows. \\

\noindent
(2) In order to estimate $\Var(X_{n,k})$, we apply the law of total variance:
\begin{align*}
\Var(X_{n,k})=\Var(\Ex(X_{n,k} \mid Z_{n,k}))+\Ex(\Var(X_{n,k}\mid Z_{n,k})).
\end{align*}
 Again as $X_{n,k}$ conditioned on $Z_{n,k}=m$ for some integer $m$ is binomially distributed with parameters $m$ and $p_k$, we find $\Ex(X_{n,k} \mid Z_{n,k})=p_kZ_{n,k}$ and $\Var(X_{n,k} \mid Z_{n,k})=Z_{n,k}p_k(1-p_k)$. Thus, we have
\begin{align*}
\Var(X_{n,k})= \Var(p_k Z_{n,k}) + \Ex(Z_{n,k}p_k(1-p_k))
= \Var(Z_{n,k})p_k^2 + \Ex(Z_{n,k})p_k(1-p_k).
\end{align*}
With $\Ex(Z_{n,k})=\frac{2n}{k(k+1)}$ and $$\Var(Z_{n,k})=\frac{2(k-1)(4k^2-3k-4)n}{(k+1)^2k(2k-1)(2k+1)},$$ (see for example \cite{FengMP08}), this yields
\begin{align*}
\Var(X_{n,k}) = \frac{2(k-1)(4k^2-3k-4)np_k^2}{(k+1)^2k(2k-1)(2k+1)} + \frac{2np_k(1-p_k)}{k(k+1)} = O\left(\frac{np_k}{k^2}\right).
\end{align*}
\\

\noindent
(3) In order to estimate $\Ex(Y_{n, \varepsilon})$, first observe that 
\begin{align*}
\Ex(Y_{n, \varepsilon}) = \sum_{k > n^{\varepsilon}}\Ex(Z_{n,k}).
\end{align*}
With $\Ex(Z_{n,k})=\frac{2n}{k(k+1)}$ for $n^{\varepsilon}<k<n$ and $\Ex(Z_{n,n})=1$, this yields
\begin{align*}
\Ex(Y_{n, \varepsilon}) = \sum_{n^{\varepsilon}<k \leq n-1}\frac{2n}{k(k+1)} +1 = \frac{2n}{\lceil n^{\varepsilon} \rceil}-1= O(n^{1-\varepsilon}).
\end{align*}

\noindent
(4) For the second part of the statement, we apply Chebyshev's inequality to obtain:
\begin{align*}
 \mathbb{P}\left(\left| X_{n,k}-\frac{2np_k}{k(k+1)}\right| \geq p_k^{1/2}n^{1/2+\varepsilon}k^{-1}\right)\leq \frac{\Var(X_{n,k})k^2}{p_kn^{1+2\varepsilon}} = O(n^{-2\varepsilon}).
\end{align*}
Hence, by the union bound, the probability that the stated inequality fails for any $k$ in the given range is $O(n^{-\varepsilon})$, proving that the given statement holds with high probability. 
Furthermore, with Markov's inequality, we find
\begin{align*}
\mathbb{P}(Y_{n, \varepsilon}>n^{1-\varepsilon/2})\leq \frac{\Ex(Y_{n, \varepsilon})}{n^{1-\varepsilon/2}}=O(n^{-\varepsilon/2}).
\end{align*}
Thus, the second inequality holds with high probability as well. \hfill \qedsymbol
\end{proof}

\section{Fringe Subtrees in Uniformly Random Binary Trees}
\subsection{Ordered Fringe Subtrees}
We provide the proof of Theorem \ref{thm:ordereduniform} first, since it is simplest and provides us with a template for the other proofs. Basically, it is a refinement of the proof for the corresponding special case of Theorem 3.1 in \cite{RalaivaosaonaW15}. In the following sections, we refine the argument further to prove Theorems \ref{thm:unordereduniform}, \ref{thm:unorderedbst} and \ref{thm:orderedbst}. 

\begin{proof}[Proof of Theorem \ref{thm:ordereduniform}]
We prove the statement in two steps: In the first step, we show that the upper bound $H_n \leq c n/\sqrt{\ln n}(1+o(1))$ holds for $c = 2\sqrt{\ln 4/ \pi}$ both in expectation and with high probability. In the second step, we prove the corresponding lower bound. 

\emph{The upper bound:} Let $k_0 = \log_4n$. The number $H_n$ of distinct fringe subtrees in a uniformly random ordered binary tree with $n$ leaves equals (i) the number of such distinct fringe subtrees of size at most $k_0$ plus (ii) the number of such distinct fringe subtrees of size greater than $k_0$. We upper-bound (i) by the number of all ordered binary trees of size at most $k_0$ (irrespective of their occurrence as fringe subtrees), which is
\begin{align*}
\sum_{k=0}^{k_0-1}C_k = O\left(\frac{4^{k_0}}{k_0^{3/2}}\right)=O\left(\frac{n}{(\ln n)^{3/2}}\right).
\end{align*}
This upper bound holds deterministically. 
Furthermore, we upper-bound (ii) by the total number of fringe subtrees of size greater than $k_0$ occurring in the tree: We apply Lemma \ref{lem:fs_lemma} with $a = 1/\ln 4$ and $\varepsilon = 1/6$ and let $\mathcal{S}_k$ denote the set $\mathcal{T}_k$, such that $s_k = C_{k-1}$, to obtain:
\begin{align*}
\left(\sum_{k_0 < k \leq n^{\varepsilon}}X_{n,k}\right) + Y_{n,\varepsilon} &= \frac{n}{\sqrt{\pi}} \sum_{k_0 < k \leq n^{\varepsilon}} k^{-3/2}\left(1+O(k^{-1})\right) + O(n^{1-\varepsilon/3})\\
&= \frac{2\sqrt{\ln 4}}{\sqrt{\pi}} \cdot \frac{n}{\sqrt{\ln n}} + O(n/(\ln n)^{3/2}),
\end{align*}
in expectation and with high probability as well, as the estimate from Lemma \ref{lem:fs_lemma} (part (4)) holds with high probability simultaneously for all $k$ in the given range. As we have
\begin{align*}
H_n \leq \sum_{k \leq k_0} C_{k-1} + \Big(\sum_{k_0 < k \leq n^{\varepsilon}}X_{n,k}\Big) + Y_{n,\varepsilon},
\end{align*}
 we can combine the two bounds to obtain the upper bound on $H_n$ stated in Theorem \ref{thm:ordereduniform}, both in expectation and with high probability.
 
\emph{The lower bound:} Again, let $k_0 = \log_4 n$ and $\varepsilon =\frac{1}{6}$. From the first part of the proof, we find that the main contribution to the total number of fringe subtrees in a uniformly random binary tree of size $n$ comes from fringe subtrees of sizes $k$ with $k_0 < k \leq n^{\varepsilon}$. Hence, in order to lower-bound the number $H_n$ of distinct fringe subtrees in a uniformly random binary tree with $n$ leaves, we only count distinct fringe subtrees of sizes $k$ with $k_0 < k \leq n^{\varepsilon}$ and show that we did not overcount too much in the first part of the proof by upper-bounding this number by the total number of fringe subtrees of sizes $k$.
To this end, let $X_{n,k}^{(2)}$ denote the number of  pairs of identical fringe subtrees of size $k$ in a uniformly random ordered binary tree of size $n$. Each such pair can be obtained as follows: Take an ordered tree with $n-2k+2$ leaves, pick two leaves, and replace them by the same ordered binary tree of size $k$. The total number of such pairs of identical fringe subtrees of size $k$ is thus
\begin{align*}
C_{n-2k+1}\cdot \binom{n-2k+2}{2}\cdot C_{k-1} = \frac{4^{n-k}}{2 \pi k^{3/2}}(n-2k+1)^{1/2}(1+O(1/k)).
\end{align*}
By dividing by $C_{n-1}$, i.e. the total number of binary trees of size $n$, we thus obtain the expected value: 
\begin{align*}
\Ex(X_{n,k}^{(2)}) =\frac{1}{C_{n-1}}\frac{4^{n-k}}{2 \pi k^{3/2}}(n-2k+1)^{1/2}(1+O(1/k)) =O(4^{-k}n^2k^{-3/2}).
\end{align*}
Thus, we find
\begin{align*}
\sum_{k_0 < k \leq n^{\varepsilon}} \Ex(X_{n,k}^{(2)}) = O\left(n^2\frac{4^{-k_0}}{k_0^{3/2}}\right)=O\left(\frac{n}{(\ln n)^{3/2}}\right).
\end{align*}
If a binary tree of size $k$ occurs $m$ times as a fringe subtree in a uniformly random binary tree of size $n$, it contributes $m - \binom{m}{2}$ to the random variable $X_{n,k}-X_{n,k}^{(2)}$. Since $m - \binom{m}{2} \leq 1$ for all non-negative integers $m$, we find that $X_{n,k}-X_{n,k}^{(2)}$ is a lower bound on the number of distinct fringe subtrees with $k$ leaves. Hence, we have
\begin{align*}
H_n \geq \sum_{k_0 < k \leq n^{\varepsilon}} X_{n,k} - \sum_{k_0 < k \leq n^{\varepsilon}}X_{n,k}^{(2)}.
\end{align*}
The second sum is $O(n/(\ln n)^{3/2})$ in expectation and thus with high probability as well by the Markov inequality. As the first sum is 
$
\frac{2\sqrt{\ln 4}}{\sqrt{\pi}} \cdot \frac{n}{\sqrt{\ln n}} (1+o(1)),
$
both in expectation and with high probability by our estimate from the first part of the proof, 
the statement of Theorem \ref{thm:ordereduniform} follows. \hfill \qedsymbol
\end{proof}

As the main idea of the proof is to split the number of distinct fringe subtrees into the number of distinct fringe subtrees of size at most $k_0$ plus the number of distinct fringe subtrees of size greater than $k_0$ for some suitably chosen integer $k_0$, this type of argument is called a \emph{cut-point argument} and the integer $k_0$ is called the \emph{cut-point} (see \cite{FlajoletGM97}).
This basic technique is applied in several previous papers to similar problems (see for instance  \cite{Devroye98}, \cite{FlajoletGM97}, \cite{RalaivaosaonaW15}, \cite{SeelbachLo18}). Moreover, we remark that the statement of Theorem \ref{thm:ordereduniform} can be easily generalized to simply generated families of trees.

\subsection{Unordered Fringe Subtrees}

In this subsection, we prove Theorem \ref{thm:unordereduniform}. For this, we refine the cut-point argument we applied in the proof of Theorem \ref{thm:ordereduniform}: In particular, for the lower bound on $F_n$, we need a result due to B\'ona and Flajolet \cite{bonaf09} on the number of automorphisms of a uniformly random ordered binary tree. It is stated for random phylogenetic trees in \cite{bonaf09}, but the two probabilistic models are equivalent.

\begin{theo}[\hspace{1sp}\cite{bonaf09}, Theorem 2]\label{thm:cltuniform}
Consider a uniformly random ordered binary tree $T_k$ with $k$ leaves, and let $A_k = \abs{\Aut(T_k)}$ be the cardinality of its automorphism group. The logarithm of this random variable satisfies a central limit theorem: For certain positive constants $\gamma$ and $\sigma_1$, we have
$$\Prob(A_k \leq 2^{\gamma k + \sigma_1 \sqrt{k} x}) \overset{k \to \infty}{\to} \frac{1}{\sqrt{2\pi}} \int_{-\infty}^x e^{-t^2/2}\,dt$$
for every real number $x$. The numerical value of the constant $\gamma$ is $0.2710416936$.
\end{theo}

With Theorem \ref{thm:cltuniform}, we are able to upper-bound the probability that two fringe subtrees of the same size are isomorphic in our proof of Theorem \ref{thm:unordereduniform}:

\begin{proof}[Proof of Theorem \ref{thm:unordereduniform}]
We prove the statement in two steps: First, we show that the upper bound on $F_n$ stated in Theorem \ref{thm:unordereduniform}  holds both in expectation and with high probability, then we prove the respective lower bound. 

\emph{The upper bound:} The proof for the upper bound in Theorem \ref{thm:unordereduniform} exactly matches the first part of the proof of Theorem \ref{thm:ordereduniform}, except that we choose a different cut-point: Let $k_0 = \log_b n$, where $b \approx2.4832535362$ is the constant in the asymptotic formula~\eqref{eq:wedder} for the Wedderburn-Etherington numbers. We then find 
\begin{align*}
F_n \leq \sum_{k < k_0} W_k + \Big( \sum_{k_0 \leq k \leq n^{\epsilon}} X_{n,k}  \Big) + Y_{n,\epsilon} = \frac{2\sqrt{\ln b}}{\sqrt{\pi}} \cdot \frac{n}{\sqrt{\ln n}} + O(n (\ln n)^{-3/2}),
\end{align*}
both in expectation and with high probability, where the estimates for $X_{n,k}$ and $Y_{n, \varepsilon}$ follow again from Lemma \ref{lem:fs_lemma}. We have $2 \sqrt{\ln b}/\sqrt{\pi} \approx 1.0761505454$.

\emph{The lower bound:} 
As a consequence of Theorem \ref{thm:cltuniform}, the probability that the cardinality of the automorphism group of a uniformly random binary tree $T_k$ of size $k$ satisfies $\abs{\Aut(T_k)} \leq 2^{\gamma k - k^{3/4}}$ tends to $0$ as $k \to \infty$. We define $\SSet_k$ as the set of ordered trees with $k$ leaves that do not satisfy this inequality, so that $s_k = |\SSet_k| = C_{k-1} (1+ o(1))$. Our lower bound is based on counting only fringe subtrees in $\SSet_k$ for suitable $k$. The reason for this choice is that we have an upper bound on the number of ordered binary trees in the same isomorphism class for every tree in $\SSet_k$. Recall that the number of possible ordered representations of an unordered binary tree $t$ with $k$ leaves is given by
$2^{k-1}/\abs{\Aut(t)}$
by the orbit-stabiliser theorem. Hence, the number of ordered binary trees in the same isomorphism class as a tree $t \in \SSet_k$ is bounded above by $2^{k-1-\gamma k + k^{3/4}}$.

Now set $k_1 = \frac{1+\delta}{1+\gamma} \log_2 n$ for some positive constant $\delta < \frac23$, and consider only fringe subtrees that belong to $\SSet_k$, where $k_1 \leq k \leq n^{\delta/2}$. By Lemma~\ref{lem:fs_lemma}, the number of such fringe subtrees in a random ordered binary tree with $n$ leaves is 
$$s_k 4^{1-k}n (1+O (k/n + s_k^{-1/2}2^{k} n^{(\delta-1)/2}) )$$
both in expectation and with high probability. Since $s_k = C_{k-1}(1+o(1))$,
%$$s_k = \frac{1}{k} \binom{2k-2}{k-1} (1+ o(1)) = \frac{1}{\sqrt{\pi k^3}}4^{k-1}(1+o(1)),$$
the number of fringe subtrees that belong to $\mathcal{S}_k$ in a random ordered binary tree of size $n$ becomes $\frac{n}{\sqrt{\pi k^3}} (1+o(1))$. We show that most of these trees are the only representatives of their isomorphism classes as fringe subtrees. To this end, we consider all fringe subtrees in $\SSet_k$ for some $k$ that satisfies $k_1 \leq k \leq n^{\delta/2}$. Let the sizes of the isomorphism classes of trees in $\SSet_k$ be $r_1,r_2,\ldots,r_{\ell}$, so that $r_1+r_2+\cdots+r_{\ell} = s_k$. By definition of $\SSet_k$, we have $r_i \leq 2^{k-1-\gamma k + k^{3/4}}$ for every $i$. Let us condition on the event that their number $X_{n,k}$ is equal to $N$ for some $N \leq n$. Each of these $N$ fringe subtrees $S_1,S_2,\ldots,S_N$ follows a uniform distribution among the elements of $\SSet_k$, so the probability of being in an isomorphism class with $r_i$ elements is $r_i/s_k$. Moreover, the $N$ fringe subtrees are also all independent. Let $X_{n,k}^{(2)}$ be the number of pairs of isomorphic trees among the fringe subtrees with $k$ leaves. We have
\begin{align*}
\Ex \big(X_{n,k}^{(2)} | X_{n,k} = N \big) &= \binom{N}{2} \sum_i \Big( \frac{r_i}{s_k} \Big)^2 \leq \frac{n^2}{2s_k^2} \sum_i r_i^2 \leq  \frac{n^2}{s_k} 2^{k-2-\gamma k + k^{3/4}}.
\end{align*}
Since this holds for all $N$, the law of total expectation yields
$$\Ex \big(X_{n,k}^{(2)}\big) \leq \frac{n^2}{s_k} 2^{k-2-\gamma k + k^{3/4}} = \sqrt{\pi}n^2 k^{3/2} 2^{-k- \gamma k + k^{3/4}} (1+o(1)).$$
Since $k \geq k_1 = \frac{1+\delta}{1+\gamma} \log_2 n$, we find that
$$\Ex \big(X_{n,k}^{(2)}\big) \leq n^2 2^{-(1+\gamma)k + O(k^{3/4})} \leq n^{1-\delta} \exp \big(O( (\ln n)^{3/4} )\big).$$
Thus
$$\sum_{k_1 \leq k \leq n^{\delta/2}} \Ex\big(X_{n,k}^{(2)}\big) \leq n^{1-\delta/2} \exp \big( O((\ln n)^{3/4}) \big) = o(n/\sqrt{\ln n}).$$
%An isomorphism class of binary trees with $k$ leaves that is represented by $m$ fringe subtrees contributes $m - \binom{m}{2}$ to the random variable $X_{n,k} - X_{n,k}^{(2)}$. Since $m - \binom{m}{2} \leq 1$ for all nonnegative integers $m$
As in the previous proof, we see that $X_{n,k} - X_{n,k}^{(2)}$ is a lower bound on the number of non-isomorphic fringe subtrees with $k$ leaves. This gives us
$$F_n \geq \sum_{k_1 \leq k \leq n^{\delta/2}} X_{n,k} - \sum_{k_1 \leq k \leq n^{\delta/2}} X_{n,k}^{(2)}.$$
The second sum is negligible since it is $o(n/\sqrt{\ln n})$ in expectation and thus also with high probability by the Markov inequality. For the first sum, a calculation similar to that for the upper bound shows that it is
$$\frac{2\sqrt{(1+\gamma)\ln 2}}{\sqrt{\pi(1+\delta)}} \cdot \frac{n}{\sqrt{\ln n}} (1+o(1)),$$
both in expectation and with high probability. Since $\delta$ is arbitrary, we can choose any constant smaller than $\frac{2\sqrt{(1+\gamma)\ln 2}}{\sqrt{\pi}} \approx 1.0591261434$ for $c_1$. \hfill \qedsymbol

\end{proof}

\section{Fringe Subtrees in Random Binary Search Trees}
In this section, we prove our results presented in Theorem \ref{thm:unorderedbst} and Theorem \ref{thm:orderedbst} on the number of distinct, respectively, non-isomorphic fringe subtrees in a random binary search tree.
In order to show the respective lower bounds of Theorem \ref{thm:unorderedbst} and Theorem \ref{thm:orderedbst}, we need two theorems similar to Theorem \ref{thm:cltuniform}: The first one shows that the logarithm of the random variable $B_k = P_{\bst}(T_k)^{-1}$, where $T_k$ denotes a random binary search tree of size $k$, satisfies a central limit theorem and is needed to estimate the probability that two fringe subtrees in a random binary search tree are identical. The second one transfers the statement of Theorem \ref{thm:cltuniform} from uniformly random binary trees to random binary search trees and is needed in order to estimate the probability that two fringe subtrees in a random binary search tree are isomorphic.
The first of these two central limit theorems is shown in \cite{Fill96}:

\begin{theo}[\hspace{1sp}\cite{Fill96}, Theorem 4.1]\label{thm:cltPbst}
Consider a random binary search tree $T_k$ with $k$ leaves, and let $B_k = P_{\bst}(T_k)^{-1} $.
The logarithm of this random variable satisfies a central limit theorem: For certain positive constants $\mu$ and $\sigma_2$, we have
\begin{align*}
\mathbb{P}\left(B_k \leq 2^{\mu k + \sigma_2 \sqrt{k}x}\right) \overset{k \to \infty}{\to} \frac{1}{\sqrt{2\pi}} \int_{-\infty}^x e^{-t^2/2}\,dt
\end{align*}
for every real number $x$. The numerical value of the constant $\mu$ is 
\begin{align*}
\mu = \sum_{k=1}^{\infty}\frac{2\log_2 k}{(k+1)(k+2)} \approx 1.7363771368.
\end{align*}

\end{theo}

The second of these two central limit theorems follows from a general theorem devised by Holmgren and Janson \cite{holmgrenj15}: Let $f: \mathcal{T} \rightarrow \mathbb{R}$ denote a function mapping an ordered binary tree to a real number. Moreover, given such a mapping $f$, define $\mathcal{F}: \mathcal{T} \rightarrow \mathbb{R}$ by
\begin{align*}
\mathcal{F}(t) = \sum_{v \in t} f(t(v)).
\end{align*}
The theorem by Holmgren and Janson states:
\begin{theo}[\hspace{1sp}\cite{holmgrenj15}, Theorem 1.14]\label{holmgrenjanson}
Let $T_k$ be a random binary search tree of size $k$. If
\begin{align*}
\sum_{k=1}^{\infty} \frac{\Var(f(T_k))^{1/2}}{k^{3/2}} < \infty, \quad
\lim_{k \rightarrow \infty} \frac{\Var(f(T_k))}{k} = 0 \quad \text{ and } \quad
&\sum_{k=1}^{\infty}\frac{\Ex(f(T_k))^2}{k^2}< \infty,
\end{align*}
then for certain constants $\nu$ and $\sigma \geq 0$, we have
$$\Ex(\mathcal{F}(T_k)) \sim \nu k \text{ and } \Var(\mathcal{F}(T_k)) \sim \sigma^2 k.$$
Moreover, if $\sigma \neq 0$, then
\begin{align*}
\mathbb{P}(\mathcal{F}(T_k) \leq  \nu k + \sigma \sqrt{k} x) \overset{k \to \infty}{\to}\frac{1}{\sqrt{2\pi}} \int_{-\infty}^x e^{-t^2/2}\,dt
\end{align*}
 for every real number $x$. In particular, we have
\begin{align*}
\nu =\sum_{k=1}^{\infty}\frac{2}{k(k+1)}\Ex(f(T_k)).
\end{align*}
\end{theo}
Note that in \cite{holmgrenj15}, the equivalent binary search model is considered that allows binary trees to have unary nodes, so that the index of summation has to be shifted in the sum defining $\nu$. Moreover, note that if we set $f(t) = \log_2(|t|-1)$ for $|t|>1$ and $f(t) =0$ otherwise, we have
\begin{align*}
\mathcal{F}(t) = \sum_{v \in t} f(t(v)) = \sum_{v \in t \atop |t(v)| >1} \log_2(|t(v)|-1) = \log_2 (P_{\bst}(t)^{-1}),
\end{align*}
by definition of $P_{\bst}$ in (\ref{def:pbst}), and thus Theorem \ref{thm:cltPbst} follows as a special case of Theorem \ref{holmgrenjanson}. This special case is also considered in Example 8.13 of \cite{holmgrenj15}.\\

As our main application of Theorem \ref{holmgrenjanson}, we transfer the statement of Theorem~\ref{thm:cltuniform} from uniformly random binary trees to random binary search trees, that is, we show that if the random number $A_k = \abs{\Aut(T_k)}$ denotes the size of the automorphism group of a random binary search tree $T_k$ with $k$ leaves, then the logarithm of this random variable satisfies a central limit theorem as well. For this, we define the function $f:\mathcal{T} \rightarrow \mathbb{R}$ in Theorem \ref{holmgrenjanson} by
\begin{align*}
f(t) = \begin{cases} 1 \quad &\text{if the root of } t \text{ is a symmetrical node,}\\
0 \quad &\text{otherwise.}\end{cases}
\end{align*} 
We thus have
\begin{align*}
\mathcal{F}(t) = \sum_{v \in t} f(t(v)) = \sym(t),
\end{align*}
that is, $\mathcal{F}(t)$ evaluates to the number of symmetrical nodes in $t$. Recall that $2^{\sym(t)}$ equals the size of the automorphism group $\Aut (t)$ of $t$. It is not difficult to check that $f$ satisfies the conditions of Theorem \ref{holmgrenjanson}: As $f(t) \in \{0,1\}$ for every $t \in \mathcal{T}$, we have $\Ex(f(T_k)^2) = \Ex(f(T_k)) \in [0,1]$ and thus $\Var(f(T_k)) \in [0,1]$ as well, so that the assumptions of Theorem \ref{holmgrenjanson} are satisfied. In order to determine the corresponding value $\nu$, we start with estimating the expectation $\Ex(f(T_k))$: 

If $k$ is odd, then $\Ex(f(T_k))=0$, as the fringe subtrees rooted at the root's children cannot be of the same size in this case and thus cannot be isomorphic. 
If $k$ is even, then $\Ex(f(T_k))$ equals the probability that these two subtrees are of the same size $\frac{k}{2}$ (which is $\frac{1}{k-1}$) times the probability that these two subtrees of size $\frac{k}{2}$ are isomorphic. 

In order to estimate the latter probability, let $P_k^r$ for positive integers $k$ and $r$ denote the probability that $2^r$ random binary search trees of size $k$ are isomorphic, and let $\delta(k) = 1$ if $k$ is even and $\delta(k) = 0$ otherwise. We find that $P_k^r$ satisfies the following recurrence relation:
\begin{align*}
P_k^r = \smashoperator{\sum_{i=1}^{\lfloor\frac{k-1}{2}\rfloor}} \Big(\frac{2}{k-1}\Big)^{\!2^r}\!P_i^rP_{k-i}^r + \delta(k) \Big(\frac{1}{k-1}\Big)^{\!2^r}\!\Big(2^{2^r-1}\big(P_{k/2}^r\big)^2\!-\!\Big(2^{2^r-1}-1\Big)P_{k/2}^{r+1}\Big),
\end{align*}
with $P_k^r = 1$ for $k \in \{1,2,3\}$ and every positive integer $r$. To see that this recurrence relation holds, first consider the case that $k$ is odd: If all the $2^r$ trees are isomorphic, then the respective sizes of the fringe subtrees rooted at the root nodes' children must coincide, that is, there are integers $i$ and $k-i$ with $1 \leq i \leq \lfloor \frac{k-1}{2} \rfloor$, such that for each of the $2^r$ trees, one of those subtrees is of size $i$ while the other is of size $k-i$. This holds with probability $(2/(k-1))^{2^r}$. Moreover, all of the $2^r$ subtrees of size $i$ (respectively, $k-i$) have to be isomorphic, which holds with probability $P_i^r$ (respectively, $P_{k-i}^r$).

If $k$ is even, we furthermore have to consider the case that $i = \frac{k}{2}$, which holds with probability $(1/(k-1))^{2^r}$: In this case pick the first of the $2^r$ trees and let $t_1$ (respectively, $t_2$) denote the fringe subtree of size $\frac{k}{2}$ rooted at the root node's left (respectively, right) child. For all of the other $2^r-1$ trees, one of the fringe subtrees rooted at the root node's children has to be isomorphic to $t_1$, while the other has to be isomorphic to $t_2$: This holds with probability $(P_{k/2}^r)^2$. Moreover, for each of those $2^r-1$ many trees, we can choose whether the subtree rooted at the root's left child or right child is isomorphic to $t_1$, which gives us $2^{2^{r}-1}$ many possibilities.
However, in the case that 
$t_1$ is isomorphic to $t_2$ as well (which means that all the $2^{r+1}$ subtrees are isomorphic, which holds with probability $P_{k/2}^{r+1}$), this means some overcounting, which is taken into account by the final term. Thus, the recursion for $P_k^r$ follows.
We find for a random binary search tree $T_k$ of size $k$:
\begin{align*}
\Ex(f(T_k)) = \begin{cases} \frac{1}{k-1}P^1_{k/2} \quad &\text{if } k \text{ is even,}\\
0 \quad &\text{otherwise.} \end{cases} 
\end{align*}
Thus, we have
\begin{align*}
\nu = \sum_{k=1}^{\infty}\frac{2}{k(k+1)}\Ex(f(T_k)) = \sum_{k=1}^{\infty} \frac{P_k^1}{k(2k+1)(2k-1)} \approx 0.3795493473,
\end{align*}
where the numerical value for $\nu$ can be determined using the recurrence relation for $P_k^r$. We remark that the constant $\sigma^2$ can also be evaluated numerically. It is approximately $0.115$, thus in particular not $0$, but we do not need its precise value.  The following theorem now follows from Theorem \ref{holmgrenjanson}:

\begin{theo}\label{thm:cltbst}
Consider a random binary search tree $T_k$ with $k$ leaves, and let $A_k =\abs{\Aut (T_k)} $ be the cardinality of its automorphism group. The logarithm of this random variable satisfies a central limit theorem: for certain positive constants $\nu$ and $\sigma_3$, we have
\begin{align*}
\mathbb{P}(A_k \leq 2^{\nu k + \sigma_3 \sqrt{k} x})  \overset{k \to \infty}{\to} \frac{1}{\sqrt{2\pi}} \int_{-\infty}^x e^{-t^2/2}\,dt
\end{align*}
for every real number $x$. The numerical value of $\nu$ is $\nu \approx 0.3795493473$.
\end{theo}

\subsection{Ordered Fringe Subtrees in Random Binary Search Trees}

We are now able to prove Theorem \ref{thm:orderedbst}:

\begin{proof}[Proof of Theorem \ref{thm:orderedbst}]
\emph{The upper bound:} Let $k_0 = \log_4 n $. We upper-bound the number $J_n$ of distinct fringe subtrees in a random binary search tree with $n$ leaves as follows: The number of distinct fringe subtrees with fewer than $k_0$ leaves is trivially bounded from above by the number of all binary trees of size at most $k_0$ (irrespective of their occurrence as fringe subtrees), which is
\begin{align*}
\sum_{k<k_0} C_{k-1} = O\left(\frac{4^{k_0}}{k_0^{3/2}}\right) = O\left(\frac{n}{(\ln n)^{3/2}}\right).
\end{align*}
This upper bound holds deterministically. The number of distinct fringe subtrees with at least $k_0$ leaves is upper-bounded by the total number of fringe subtrees with at least $k_0$ leaves: For this, we apply Lemma \ref{lemma1bst} with $a =\frac{1}{\ln 4}$, $\varepsilon< \frac{1}{3}$ and $\mathcal{S}_k = \mathcal{T}_k$, so that $p_k=1$ for $k_0 \leq k \leq n^{\varepsilon}$. Thus, both in expectation and with high probability, as the estimate from Lemma \ref{lemma1bst} (part (4)) holds with high probability simultaneously for all $k$ in the given range, we obtain:
\begin{align*}
\left(\sum_{k_0 \leq k \leq n^{\varepsilon}} X_{n,k}\right) + Y_{n,\varepsilon} &= \sum_{k_0 \leq k \leq n^{\varepsilon}} \frac{2n}{k(k+1)}(1+O(kn^{-1/2+\varepsilon})) + O(n^{1-\frac{\varepsilon}{2}}) \\ &= 2\ln 4  \cdot \frac{n}{\ln n }(1+o(1)).
\end{align*}
Hence, we find that $J_n$ is both in expectation and with high probability bounded from above by
\begin{align*}
J_n \leq \sum_{k < k_0} C_k +\left( \sum_{k_0 \leq k \leq n^{\varepsilon}} X_{n,k}\right) + Y_{n,\varepsilon} =2 \ln 4 \cdot \frac{n}{\ln n} (1+o(1)).
\end{align*}
The numerical value of the constant is $c_6=2\ln(4) \approx 2.7725887222$.

\emph{The lower bound:} As a consequence of Theorem \ref{thm:cltPbst}, the probability that $P_{\bst}(T_k)^{-1} \leq 2^{ \mu k -k^{3/4}}$ for a random binary search tree $T_k$ of size $k$ tends to $0$ for $k \rightarrow \infty$. Let $\mathcal{S}_k$ denote the set of binary trees of size $k$ that do not satisfy this inequality: Thus, every binary tree $t \in \mathcal{S}_k$ satisfies $P_{\bst}(t) \leq 2^{- \mu k + k^{3/4}}$ and we have $p_k = 1+o(1)$.
In order to prove the lower bound, we only consider fringe subtrees in $\mathcal{S}_k$ for suitable $k$: Thus, we can suitably upper-bound the probability that two fringe subtrees of size $k$ in a random binary search tree are identical. Let  $\delta$ denote a positive constant with $\delta < \frac23$, let $k_1 = (1+\delta)\mu^{-1}\log_2 n$ and let $k_1 \leq k \leq n^{\delta/2}$.
By Lemma \ref{lemma1bst}, the number of such fringe subtrees in a random binary search tree with $n$ leaves is $2np_k/(k(k+1))$ in expectation and
\begin{align*}
\frac{2np_k}{k(k+1)}(1+O(p_k^{-1/2}n^{(\delta-1)/2}k))
\end{align*}
with high probability. Furthermore, let $X^{(2)}_{n,k}$ denote the (random) number of pairs of identical fringe subtrees among the fringe subtrees with $k$ leaves that belong to $\mathcal{S}_k$, for $k_1 \leq k \leq n^{\delta/2}$. Let us condition on the event that $X_{n,k}=N$ for some nonnegative integer $N \leq n$. Those $N$ fringe subtrees are all independent random binary search trees, and the probability that such a fringe subtree equals a given binary tree $t \in \mathcal{S}_k$ is $P_{\bst}(t)/p_k$. Thus, we have
\begin{align*}
\Ex(X_{n,k}^{(2)} \mid X_{n,k}=N) = \binom{N}{2} \sum_{t \in \mathcal{S}_k} \left(\frac{P_{\bst}(t)}{p_k}\right)^2 \leq \frac{n^2}{2}\frac{1}{p_k^2}\sum_{t \in \mathcal{S}_k}P_{\bst}(t)^2.
\end{align*}
As by assumption, $P_{\bst}(t) \leq 2^{-\mu k + k^{3/4}}$ for every $t \in \mathcal{S}_k$, the expected value is upper-bounded by
\begin{align*}
\Ex(X_{n,k}^{(2)} \mid X_{n,k}=N) \leq \frac{n^2}{2}\frac{2^{-\mu k +k^{3/4}}}{p_k^2}\sum_{t \in \mathcal{S}_k}P_{\bst}(t) = \frac{n^2}{p_k} 2^{-\mu k-1 +k^{3/4}}.
\end{align*}
Since this upper bound for $\Ex(X_{n,k}^{(2)} \mid X_{n,k}=N)$ holds independently of $N$, the law of total expectation yields
\begin{align*}
\Ex(X_{n,k}^{(2)}) = \sum_{N=0}^n\Ex(X_{n,k}^{(2)} \mid X_{n,k} = N)\mathbb{P}(X_{n,k} = N) \leq \frac{n^2}{p_k} 2^{-\mu k-1 +k^{3/4}}.
\end{align*}
With $p_k = 1+o(1)$, we obtain
\begin{align*}
\Ex(X_{n,k}^{(2)}) \leq n^2 2^{-\mu k-1 +k^{3/4}}(1+o(1)) =  n^2 2^{-\mu k +O(k^{3/4})}.
\end{align*}
As $k \geq k_1= \frac{1+\delta}{\mu} \log_2 n$, we find
\begin{align*}
\Ex(X_{n,k}^{(2)}) \leq n^2 2^{-(1+\delta)\log_2 n +O(k^{3/4})} \leq n^{1-\delta} 2^{O((\log_2 n)^{3/4})}.
\end{align*}
Thus,
\begin{align*}
\sum_{k_1 \leq k \leq n^{\delta/2}}\Ex(X_{n,k}^{(2)}) \leq n^{1-\delta/2}2^{O((\log_2 n)^{3/4})} = o(n/\ln n).
\end{align*}

The (random) number $J_n$ of distinct fringe subtrees in a random binary search tree of size $n$ is lower-bounded by the number of distinct fringe subtrees of sizes $k$ for $k_1 \leq k \leq n^{\delta/2}$ that belong to $\mathcal{S}_k$, and this number is again lower-bounded by the sum over $X_{n,k}-X_{n,k}^{(2)}$ for $k_1 \leq k \leq n^{\delta/2}$. 
We thus have
\begin{align*}
J_n \geq \sum_{k_1 \leq k \leq n^{\delta/2}} X_{n,k} - \sum_{k_1 \leq k \leq n^{\delta/2}}X_{n,k}^{(2)}.
\end{align*}
The second sum is $o(n/\ln n)$ in expectation and hence by Markov's inequality with high probability as well. The first sum can be estimated using Lemma \ref{lemma1bst}, as in the proof of the upper bound, which yields
\begin{align*}
\sum_{k_1 \leq k \leq n^{\delta/2}} X_{n,k} = \frac{2 \mu \ln 2}{1+\delta} \cdot \frac{n}{\ln n}(1+o(1)),
\end{align*}
in expectation and with high probability.
Since $\delta$ can be chosen arbitrarily, the desired statement holds for any constant 
\begin{align*}
c_5 <2 \mu \ln 2 \approx 2.4071298335.
\end{align*}
\flushright{\qedsymbol}
\end{proof}

\subsection{Unordered Fringe Subtrees in Random Binary Search Trees}
It remains to prove Theorem \ref{thm:unorderedbst}:

\begin{proof}[Proof of Theorem \ref{thm:unorderedbst}]
\emph{The upper bound}: The proof for the upper bound exactly matches the first part of the proof of Theorem \ref{thm:orderedbst}, except that we choose the cut-point $k_0 = \log_b(n)$, where $b \approx 2.4832535362$ is the constant determining the asymptotic growth of the Wedderburn-Etherington numbers.

\emph{The lower bound}:
Let $T_k$ denote a random binary search tree with $k$ leaves.
As a consequence of Theorem \ref{thm:cltbst}, the probability that $\abs{\Aut(T_k)}=A_k \leq 2^{\nu k - k^{3/4}}$ tends to $0$ as $k \rightarrow \infty$. Moreover, by Theorem \ref{thm:cltPbst} the probability that $B_k  = P_{\bst}(T_k)^{-1} \leq 2^{\mu k -k^{3/4}}$ tends to $0$ for $k \rightarrow \infty$. Let $\mathcal{S}_k$ denote the set of ordered binary trees with $k$ leaves for which neither of the two inequalities is satisfied: Thus, every binary tree $t \in \mathcal{S}_k$ satisfies $P_{\bst}(t) \leq 2^{- \mu k +k^{3/4}}$ and $\abs{\Aut (t)} \geq 2^{\nu k -k^{3/4}}$, and we have $p_k = 1+o(1)$. By the orbit-stabilizer theorem, we find that the number of ordered binary trees in the same isomorphism class as a tree $t \in \mathcal{S}_k$ is bounded from above by
\begin{align*}
\frac{2^{k-1}}{\abs{\Aut (t)}} \leq \frac{2^{k-1}}{2^{\nu k -k^{3/4}}} = 2^{(1- \nu) k -1+ k^{3/4}}.
\end{align*}

In order to prove the lower bound, we only consider fringe subtrees in $\mathcal{S}_k$ for suitable $k$: Thus, we are able to suitably upper-bound the probability that two fringe subtrees of size $k$ in a random binary search tree are identical as unordered binary trees. 
Let $\delta < \frac{2}{3}$, let $k_1 = (1 + \delta)\log_2 n/(\mu + \nu -1)$ and let $k_1 \leq k \leq n^{\delta/2}$. 
By Lemma \ref{lemma1bst}, the number of fringe subtrees of size $k$ with $k_1 \leq k \leq n^{\delta/2}$ is $2np_k/(k(k+1))$ in expectation and 
\begin{align*}
\frac{2np_k}{k(k+1)}(1+O(p_k^{-1/2}n^{(\delta-1)/2}k))
\end{align*}
with high probability.
For $k_1 \leq k \leq n^{\delta/2}$, let $X_{n,k}^{(2)}$ denote the (random) number of pairs of isomorphic binary trees among the fringe subtrees of size $k$ that belong to $\mathcal{S}_k$. Moreover, let $l$ denote the number of isomorphism classes of binary trees in $\mathcal{S}_k$ and for each isomorphism class, pick one representative: Let $t_1, t_2, \dots, t_l$ denote those representatives. If a binary tree $t$ is in the same isomorphism class as tree $t_i$, then $P_{\bst}(t) = P_{\bst}(t_i)$ and $\abs{\Aut (t)} =\abs{\Aut (t_i)}$. In particular, if $t \in \mathcal{T}_k$ is isomorphic to a representative $t_i$, then $t \in \mathcal{S}_k$ as well, that is, all binary trees that are isomorphic to a binary tree in $\mathcal{S}_k$ are automatically contained in $\mathcal{S}_k$ as well.
As there are $2^{k-1}/\abs{\Aut t_i}$ many trees in the same isomorphism class as the binary tree $t_i$, we find
\begin{align*}
\sum_{i=1}	^l P_{\bst}(t_i)\frac{2^{k-1}}{\abs{\Aut t_i}} = \sum_{t \in \mathcal{S}_k} P_{\bst}(t) =p_k.
\end{align*}
Let us condition on the event that $X_{n,k}=N$ for some integer $0 \leq N \leq n$. Those $N$ fringe subtrees are all independent random binary search trees, and the probability that such a fringe subtree is isomorphic to a given binary tree $t_i \in \mathcal{S}_k$ is $(P_{\bst}(t_i)/p_k) \cdot (2^{k-1}/\abs{\Aut (t_i)})$. Thus, we find
\begin{align*}
\Ex(X_{n,k}^{(2)} \mid X_{n,k} = N) \!= \!\binom{N}{2} \sum_{i=1}^l\left(\frac{2^{k-1}P_{\bst }(t_i)}{\abs{\Aut t_i}p_k}\right)^2 \!\leq \frac{n^2}{2}\frac{1}{p_k^2}\sum_{i=1}^l\left(\frac{2^{k-1}P_{\bst}(t_i)}{\abs{\Aut t_i}}\right)^2\!.
\end{align*}
As $t_1, \dots, t_l \in \mathcal{S}_k$, we have $P_{\bst}(t_i)\leq 2^{-\mu k + k^{3/4}}$ and thus $$2^{k-1}/\abs{\Aut t_i} \leq 2^{(1-\nu)k -1+k^{3/4}}$$ for every $i \in \{1, 2, \dots, l\}$. Hence
\begin{align*}
\Ex(X_{n,k}^{(2)} \mid X_{n,k} = N) &\leq \frac{n^2 2^{(1-\nu - \mu)k + 2k^{3/4}-2}}{p_k^2}\sum_{i=1}^l \left(\frac{2^{k-1}P_{\bst}(t_i)}{\abs{\Aut t_i}}\right) \\&= p_k^{-1}n^2 2^{(1-\nu - \mu)k + 2k^{3/4}-2}.
\end{align*}
As this upper bound on the expectation is independent of $N$, we find by the law of total expectation:
\begin{align*}
\Ex(X_{n,k}^{(2)}) \leq p_k^{-1}n^2 2^{(1-\nu - \mu)k + 2k^{3/4}-2} = n^2 2^{-(\nu + \mu-1)k + 2k^{3/4}-2}(1+o(1)).
\end{align*}
With $k \geq k_1 = (1+\delta)/(\mu + \nu -1) \log_2 n$, we obtain
\begin{align*}
\Ex(X_{n,k}^{(2)}) \leq n^2 2^{-(1+\delta) \log_2 n + O((\log_2 n)^{3/4})} = n^{1-\delta} 2^{O((\log_2 n)^{3/4})}.
\end{align*}
Thus,
\begin{align*}
\sum_{k_1 \leq k \leq n^{\delta/2}} \Ex(X_{n,k}^{(2)}) \leq n^{1 - \frac{\delta}{2}}2^{O((\log_2 n)^{3/4})} = o(n/\ln n).
\end{align*}
Analogously as in the previous proofs, we lower-bound the random number $G_n$ of non-isomorphic fringe subtrees in a random binary search tree of size $n$ by the number of such fringe subtrees of sizes $k$ for $k_1 \leq k \leq n^{\delta/2}$ that belong to $\mathcal{S}_k$, and this number is again lower-bounded by the sum over $X_{n,k}-X_{n,k}^{(2)}$ for $k_1 \leq k \leq n^{\delta/2}$
by the inclusion-exclusion principle. We thus have
\begin{align*}
G_n \geq \sum_{k_1 \leq k \leq n^{\delta/2}} X_{n,k} - \sum_{k_1 \leq k \leq n^{\delta/2}} X_{n,k}^{(2)}.
\end{align*}
The second sum is $o(n/\ln n)$ in expectation and hence by Markov's inequality with high probability as well. The first sum is bounded similarly as in the estimate for the upper bound, which yields
\begin{align*}
\sum_{k_1 \leq k \leq n^{\delta/2}} X_{n,k} = \frac{2 (\mu + \nu -1)\ln 2}{1+\delta} \cdot \frac{n}{\ln n}(1+o(1)).
\end{align*}
Since $\delta$ can again be chosen arbitrarily, the desired statement holds for any constant
$
c_3 < 2(\mu + \nu -1)\ln 2 \approx 1.5470025923.$ \hfill \qedsymbol

\end{proof}

\section{Open Problems}

The following natural question arises from our results: Is it possible to determine constants $\alpha_1, \alpha_2, \alpha_3$ with $c_1 \leq \alpha_1 \leq c_2$, $c_3 \leq \alpha_2 \leq c_4$ and $c_5 \leq \alpha_3 \leq c_6$, such that
\begin{align*}
\Ex(F_n) =  \frac{\alpha_1 n}{\sqrt{\log n}}(1+o(1)), \  \Ex(G_n) =  \frac{\alpha_2 n}{\log n}(1+o(1)), \ \Ex(J_n) =  \frac{\alpha_3 n}{\log n}(1+o(1)),
\end{align*}
respectively, and
\begin{align*}
\frac{F_n}{n/\sqrt{\log n}} \overset{P}{\to} \alpha_1, \
\frac{G_n}{n/\log n} \overset{P}{\to} \alpha_2, \ \text{and} \
\frac{J_n}{n/\log n} \overset{P}{\to} \alpha_3\  ?
\end{align*}
In order to prove such estimates, it seems essential to gain a better understanding of the random variables $P_{\bst}(T_k)^{-1}$ and $\abs{\Aut(T_k)}$, in particular their distributions further away from the mean values, for random binary search trees or uniformly random ordered binary trees $T_k$ of size $k$.
\bibliographystyle{plain}
\bibliography{bib}

\end{document}